\documentclass[11pt]{amsart}
\usepackage{graphicx}
\usepackage{amssymb}
\usepackage{epstopdf}
\usepackage{amsmath}
\usepackage{amsthm}
\usepackage{subfigure}
\usepackage{dsfont}
\usepackage{color}
\newtheorem{theorem}{Theorem}[section]
\newtheorem{lemma}{Lemma}[section]
\newtheorem{proposition}{Proposition}[section]

\newtheorem{definition}[theorem]{Definition}
\newtheorem{algorithm}[theorem]{Algorithm}

\theoremstyle{remark} 
\newtheorem{remark}{Remark}

\newcommand{\R}{\mathbb{R}}
\newcommand{\argmin}{\operatornamewithlimits{argmin}}

\title[State Constrained Optimization and Generalized Gradients]{State Constrained Optimization with Partial Differential Equations via Generalized Gradients}
\author{Richard Barnard, Martin Frank, Michael Herty}
\date{\today}                                           

\begin{document}
\maketitle
\begin{abstract}
The work here considers optimization problems constrained by partial differential equations (PDEs) with additional constraints placed on the solution of the PDEs. We develop a framework using infinite-valued penalization functions and Clarke subgradients and apply this to problems with box constraints as well as more general constraints arising in applications, such as constraints on the average value of the state in subdomains. The framework also allows for problems with discontinuous data in the constraints.  Numerical results of this algorithm are presented for the elliptic case and compare with other state-constrained algorithms.
\end{abstract}

\section{Introduction}
In recent years there has been intense research in the 
discussion of constrained minimization problems with 
partial differential equations. In particular, in the 
context of state constraints different methods have been
proposed and analysed, see for example
\cite{HintermullerHinze2006,BergouniouxKunisch1997,
KunischRoesch2002,Troeltzsch2005,cassas1986,HintermullerItoKunisch2003,MeyerRoeschTroeltzsch2006,GriesseVolkwein2005,SachsVolkwein2002,KelleySachs1999,Tro02,pre05055563,1103.90072}.
We refer in particular to \cite{HintermullerKunisch2006,1080.90074,1121.49030,HinPinUlb09} for a survey on state constrained elliptic problems as appearing in the numerical results later on. The latter includes a discussion of the specific case of linear elliptic problems with general constraints and related issues regarding regularity. Due to the possibly low regularity of the arising multipliers in the several cases we consider purely primal methods for solving state constraint problems. 

Our interest is in using the calculus associated with the proximal subgradient and the Clarke subgradient to derive optimality conditions for a variety of state-constrained optimization problems which are more general than standard box-constrained problems, with motivations from several applications, particularly that of radiation therapy treatment planning problems (such as those described in \cite{NieUriGoi92}).  We use exact penalization methods without smoothing, leading to the need for elements of nonsmooth analysis.  The constraints can include requirements on both the weighted average value the state takes on a given subdomain and the portion of a subdomain where the state exceeds acceptable limits.  For the sake of generality, we do not consider the question of constraint qualifications, as these can be highly dependent on the type of PDE being considered as well as specifics for the state constraint.  In the same vein, we consider problems with discontinuous data in the state constraints, further suggesting a focus on primal methods.  

For the case
of infinite dimensional problems there have been some studies on primal methods relying on penalization and smoothing  for example in \cite{GugatHerty,grossmann,GrossmannZadlo}. The method we use here can be contrasted with an approach presented in \cite{S02} for finite--dimensional linear--quadratic problems and extended  in \cite{GugatHerty} to the infinite dimensional case. This approach is based on smooth approximations to the exact $l_1$-penalty function and has also been studied in \cite{grossmann,GrossmannZadlo} 
for a finite number of constraints and other choices of smoothing and penalization update. In the finite--dimensional case, convergence for the smoothed penalty approach has been discussed in \cite{CG} under the assumption of MFCQ. This has been extended to infinite dimensions in \cite{GugatHerty} for the assumption of a strictly feasible point \cite{GugatHerty2}. Recent  publications
 \cite{1103.90072,pre05055563,HintermullerKunisch2006,1121.49030}  treat other regularization strategies {\em not based } on exact penalization. 
\par
Other existing methods using primal {\em and}Ê dual variables 
also rely on non--differentiable functions, e.g., 
the semi--smooth Newton method. Here, Newton's
method is applied to the first--order optimality system containing
a non--differentiable function by reformulation of possible box--type
state constraints. The Clarke subgradient information on the reformulated
optimality system is then used within the descent step of Newton's method. 
This has been treated
for example in \cite{HintermullerUlbrich2004,HintermullerItoKunisch2003,Conn1973,MayneMaratos1979}. 

Our goal is to provide a flexible framework for a variety of state constraints as opposed to focusing on methods tuned to a specific type of constraint (such as box constraints).  The constraints investigated are by no means exhaustive but do provide an idea on the procedure in obtaining optimality conditions for other state constraints.  The arising numerical methods may not be as efficient in the case of box constraints when compared with methods mentioned above.  However, in the case of weighted integral constraints, black box methods perform quite poorly whereas the resulting nonsmooth calculus allows for significant reduction in the dimension of the resulting subproblem.   
After giving in Section \ref{sec:problem} the description of the problems under consideration, we move to introducing the necessary nonsmooth calculus in Section \ref{sec:nonsmooth}.  We then present the relevant necessary optimality conditions for the considered problems in Section \ref{sec:optimality} before providing a brief description of the descent algorithm possible with these conditions.  Numerical examples follow in Section \ref{sec:numerics} for elliptic problems with box constraints and weighted integral constraints with a brief description of the performance of the algorithms when compared with black-box algorithms as well as the smoothed over-penalization method mentioned above.  

\section{Problem Description}
\label{sec:problem}
Let $\Omega\subset\R^{n}$ be a convex, bounded domain, a function $\overline{\psi}\in L^{2}(\Omega),$ and scalar $\alpha>0$ be given.  The theory that follows applies to  general cost functions; specifically, we only require subdifferential regularity (discussed in Section \ref{sec:nonsmooth}) and Lipschitz continuity of the reduced cost functional defined below for many of the results.  However, for notational simplicity, we focus on minimizing the standard tracking functional $J:L^{2}(\Omega)\times L^{2}(\Omega)\rightarrow\R$ given by
\begin{equation*}
J(\psi ,q):=\frac{1}{2}||\psi-\overline{\psi}||_2^2+\frac{\alpha}{2}||q||_2^2,
\end{equation*}
subject to
\begin{equation}
\label{eq:stateeq}
\mathcal{E}q=\psi,
\end{equation}
where  $\mathcal{E}:L^{2}(\Omega)\rightarrow L^{2}(\Omega)$ denotes the operator mapping a control $q$ to the state $\psi$ which solves the relevant PDE.  Where convenient, we will also write the PDE constraint \eqref{eq:stateeq} as
\begin{equation*}
A\psi-q=0
\end{equation*}
for the relevant operator $A.$  We restrict our attention to problems where the solution operator $\mathcal{E}$ is linear and bounded; however, much of the subsequent analysis requires only the strict differentiability of this operator to hold.  We consider the following three example constraints; in each we have as data constants $a<b$:
\begin{itemize}
\item {\bf Box constraints:} We require that the solution to the relevant PDE satisfy the standard box constraints $a\leq \psi(x) \leq b$ for almost all $x\in\Omega.$
\item {\bf Weighted Integral Constraints:}  We instead require the solution to satisfy
\begin{equation*}
a\leq\langle w,\psi\rangle_{L^{2}(\Omega)}\leq b
\end{equation*}
for a given function $w\in L^{2}(\Omega).$ This is motivated by the requirement in radiotherapy that a given portion of the body be given an average dose \cite{NieUriGoi92}. 
\item {\bf Total coverage constraints:}  A further type of constraint is the following.  Given $0<a<b$ and $0<c< 1,$ and a compactly contained closed subset of nonempty interior $Z\subset \Omega,$ we require that the state satisfy
\begin{equation*}
a\leq\psi(x)\leq b
\end{equation*}
for at least the fraction $c$ of the total area in $Z$.  In other words, we require
\begin{equation*}
c|Z|\leq\int_{Z}\mathds{1}_{[a,b]}\big(\psi(x)\big)dx
\end{equation*}
where $\mathds{1}_{[a,b]}$ is the standard characteristic function associated to the interval $[a,b]$.  This example can be seen as requiring that a critical region has a minimal/maximal level of whatever the state represents.  It is also motivated from radiation therapy, where one may require that a sufficient portion of a critical organ not receive a radiative dose above a given safe level as discussed in \cite{NieUriGoi92}. 
\end{itemize}
This list is not exhaustive; for instance, we can also constrain--in a combination of the second and third examples, $\langle w,R\psi\rangle$-- as well as encoding constraints on the maxima of several quantities dependent on the state.  The analysis that follows can be extended to much more general constraints as long as the appropriate operations acting on the state are either directionally Lipschitz or strictly differentiable (as applicable in the calculus rules described below).  We also note that these requirements often do not involve second-order differentiability. 

For each problem, we will use the standard indicator function $I_S:X\rightarrow\R\cup\{+\infty\}$ associated with a set $S\subset X$ for a Banach space $X:$
\begin{equation*}
I_S(x)=\begin{cases}
0,&x\in S\\
+\infty,&x\not\in S.
\end{cases}
\end{equation*}
Naturally, this is a convex function when the set $S$ is a convex set.  This also means that it is convex when $S$ is convex. Equipped with this, we create for each type of constrained problem an augmented cost functional. For the case of the box constrained problem, we look to minimize the function $J_{B}:L^{2}(\Omega)\times L^{2}(\Omega)\rightarrow\R\cup\{+\infty\}$ given by
\begin{equation}
\label{eq:boxfunc}
J_{B}(\psi,q):=\frac{1}{2}||\psi-\overline{\psi}||^{2}_{2}+\frac{\alpha}{2}||q||^{2}_{2}+\int_{\Omega}I_{[a,b]}\big(\psi\big)(x)dx
\end{equation}
where the last term takes values of either $0$ or $+\infty$ depending on if the state violates the constraint on a set of positive measure under the constraint
\begin{equation*}
A\psi-q=0.
\end{equation*}
As opposed to a standard Moreau-Yoshida regularization, the penalization here means that the minimizer of $J_{B}$ is also a feasible minimizer of the original constraints.

For the other problems, we write the optimization problem in terms of a reduced cost functional.  Thus, the weighted integral constrained problem involves minimizing
\begin{equation}
\label{eq:wifunc}
j_{WI}(q):=\frac{1}{2}||\mathcal{E}(q)-\overline{\psi}||^{2}_{2}+\frac{\alpha}{2}||q||^{2}_{2} +I_{[a,b]}(\langle w,\mathcal{E}q\rangle).
\end{equation}
Finally, the total coverage constrained problem is rewritten as minimization of
\begin{equation}
j_{TC}(q):=\frac{1}{2}||\mathcal{E}(q)-\overline{\psi}||^{2}_{2}+\frac{\alpha}{2}||q||^{2}_{2}+I_{(1-c)|Z|}\Big(\int_{Z}\mathds{1}_{(-\infty,a)\cup(b,\infty)}\big(\mathcal{E}q\big)dx\Big)
\end{equation}
which we write in this manner so that the integrand in the penalization term is lower semicontinuous.   We note that $j_{B}$ and $j_{WI}$ are convex functionals while $j_{TC}$ is not.  We assume that the choice of $a$ and $b$ guarantees that there is at least one admissible control such that the resulting state satisfies the appropriate constraints.  

\section{Tools from Nonsmooth Analysis}
\label{sec:nonsmooth}
In this section we recall the Clarke subdifferential and some of the related calculus rules which are proven in \cite{Cla90} and \cite{ClaLedSte98}.  As our focus is on (possibly) nonconvex, extended-value functions on $L^{2}(\Omega),$ we construct the Clarke subdifferential from the related proximal subdifferential, which is defined in a Hilbert space setting.
\begin{definition}
Let $S\subset X$ be a closed subset of a Hilbert space $X.$  Then we say that, for $x\not\in S,$ the vector $x-s$ is a {\em proximal normal direction} to $S$ at $s\in S$ iff $\{s\}\subset S\cap\overline{B}_{||x-s||}(x)$ and $S\cap B_{||x-s||}(x)=\emptyset.$  The collection of all nonnegative multiples $\zeta=t(x-s),~t\geq0$ where $x-s$ is a proximal normal direction at $s\in S$ is denoted $N^P_S(s)$ and called the {\em proximal normal cone} to $S$ at $s.$
Equivalently, the proximal normal cone to $S$ at $x$ is the set of vectors $\zeta$ such that for any $\delta>0,$ there is a $\sigma_{\zeta,x}>0$ such that
\begin{equation*}
\langle \zeta,y-x\rangle\leq\sigma_{\zeta,x}||y-x||^2,~ \forall y\in S\cap \overline{B}_{\delta}(x).
\end{equation*}
\end{definition}
\begin{definition}
Let $f:X\rightarrow\R\cup\{+\infty\}$ be a lower semicontinuous function and let $f(x)<\infty.$  Then the vector $\zeta\in X$ is a {\em proximal subgradient} at $x$ if 
\begin{equation*}
(\zeta,-1)\in N^P_{\text{epi }f}(x,f(x))
\end{equation*}
where $\text{epi }f$ denotes the epigraph of $f$.  The collection of all such $\zeta$ is denoted as $\partial_P f(x)$ which we call the {\em proximal subdifferential}.
Equivalently, $\zeta\in\partial_P f(x)$ if and only if there exists $\sigma_{\zeta,x}>0$ such that
\begin{equation*}
f(y)\geq f(x)+\langle\zeta,y-x\rangle+\sigma_{\zeta,x}||y-x||^2,\forall y\in B_\delta(x).
\end{equation*}
\end{definition}
From this, we construct the Clarke subdifferential and normal cone:

\begin{definition}
Let $f:X\rightarrow\R\cup\{+\infty\}$ and $x\in X$ be as in the above definition.  Furthermore, assume $f$ is Lipschitz. We define the {\em Clarke subdifferential}, denoted by $\partial_{C}f(x)$ as the closed convex hull of the set
\begin{equation*}
\partial_{L}f(x):=\{\zeta|\exists x_{i}\rightarrow x,\zeta_{i}\in \partial_{P}f(x_{i}),\zeta_{i}\xrightarrow{w}\zeta\}.
\end{equation*}

\end{definition}
\begin{definition}
For a closed set $S,$ We denote by $N^{C}_{S}(x)$ the closed convex hull of the set of vectors obtained by taking weak limits of sequences of $\zeta_{i}\in N^{P}_{S}(x_{i})$ where $x_{i}\in S$ and $x_{i}\rightarrow x.$
\end{definition}
This definition of the Clarke subdifferetial, discussed in the finite dimensional case in \cite{Roc79} and in the general Hilbert setting in \cite{ClaLedSte98}, is equivalent to the definitions in \cite{Cla90}.  However, at times we will make use of the proximal constructions in the presence of non-Lipschitz data.  We note that one can extend the Clarke subdifferential to non-Lipschitz functions but is not necessary for our analysis.  Finally, we will make use of Clarke's generalized directional gradient: 
\begin{definition}
Let $f:\R^{n}\rightarrow\R$ and $x\in\R^{n}.$ Then the generalized directional gradient of $f$ at $x$ in direction $v$ is given as
\begin{equation*}
f^{\circ}(x;v):=\limsup_{\substack{(y,r)\downarrow_{f}x\\ t\downarrow 0}}\inf_{w\rightarrow v}\frac{f(y+t w)-r}{t}
\end{equation*}
where $(y,r)\downarrow_{f}x$ means $y\rightarrow x,~r\rightarrow f(x)$ with $(y,r)\in \text{epi}f$.
\end{definition}
We also note that in the case of convex functions (respectively, sets), the proximal and Clarke subdifferentials (respectively, normal cones) coincide with the usual notions from convex analysis. Furthermore, if $x^{*}$ locally minimizes $f,$ then $0\in\partial_{C}f(x^{*});$ the converse holds if $f$ is convex: if $0\in\partial_{C}f(x^{*}),$ then $x^{*}$ is a global minimizer (see, for instance, Proposition 4.3 of Chapter 2 of \cite{ClaLedSte98}).  

In order to make use of the calculus rules for the Clarke subdifferential, we recall two additional definitions.
\begin{definition}
For two Banach spaces $X,Y$ we say $F:X\rightarrow Y$ is \emph{strictly differentiable} at $x$ if, there exists $D_s F(x)\in L(X,Y)$ such that for any $v$
\begin{equation*}
\lim_{y\rightarrow x,t\downarrow0}\frac{F(y+tv)-F(y)}{t}=\langle D_sF(x),v\rangle
\end{equation*}
and that the convergence is uniform for $v$ in compact sets.  
\end{definition}
This is a Hadamard-type strict derivative.  Similar versions exist for Gateaux-type and Fr\'echet-type strict differentiability.We will only focus on Hadamard-type strict derivatives for the remainder of the paper, and thus only refer to such $D_{s}F$ as the strict derivative.
\begin{definition}
Let $f:X\rightarrow \R \cup \{+\infty\}$ be an extended real-valued function.  We say $f$ is {\em directionally Lipschitz at $x$ with respect to $v\in X$} if $f(x)<\infty$ and
\begin{equation*}
\limsup_{(y,r)\downarrow_f x, w\rightarrow v,t\downarrow 0} \frac{f(y+tw)-r}{t}
\end{equation*}
is finite.   We say $f$ is {\em directionally Lipschitz at} $x$ if $f$ is directionally Lipschitz at $x$ with respect to at least one $v.$
\end{definition}
Theorem 1 of \cite{Roc79} states, in particular, that a lower semicontinuous function $f:X\rightarrow \R,$ for some Hilbert space X, is directionally Lipschitz at $x$ if $f(x)$ is finite and Lipschitzian at $x.$  Alternatively, and perhaps more useful in this context, the function is directionally Lipschitz if it is convex and bounded on an open set (not necessarily containing the point under consideration).  

Equipped with these two definitions, we state calculus rules for Clarke subdifferentials when applied to extended-value functions.  For proofs, we refer the reader to \cite{Cla90}.  We also note that more complex and weaker rules exist for the proximal subdifferential, which may be required in the future when dealing with more general constraints (such that the associated cost functions are merely lower-semicontinuous). 
\begin{lemma}
\label{lem:chainrule}
Let $f=g\circ F$ where $F:X\rightarrow Y$ is a strictly differentiable map between Banach space $X$ and Hilbert space $Y$ and $g:Y\rightarrow\R\cup\{+\infty\}.$  If $g$ is finite and directionally Lipschitz at $F(x)$ with
\begin{equation*}
D_s F(X)\cap\text{int}\{v:g^\circ(F(x);v)<\infty\}\neq\emptyset,
\end{equation*}
then we have 
\begin{equation*}
\partial_C f(x)\subseteq D_s F(x)^*\circ\partial_C g(F(x)).
\end{equation*}
Equality holds if $g$ is a convex, extended value function or is $C^{2}$ in the Fr\'echet sense.
\end{lemma}
\begin{proof}
This  is proven as Theorem 3 of \cite{Roc79}. 
\end{proof}
The condition for equality is a specific case of subdifferential regularity, termed by Rockafellar in \cite{Roc79}. There are more general conditions where equality can hold, as discussed in that work; another possibly salient characterization is the case where $g$ is the indicator function of a set whose tangent and contingent cones coincide.  However, for our purposes here, we will not need them.  We also will make use of the  following sum rule proven as Corollary 2 of \cite{Roc79}.
\begin{proposition}
For any finite collection of  $f_i:X\rightarrow\R$, where $X$ is a Banach space, and a point $x$ where all $f_{i}(x)<+\infty,$ we have if all but one is Lipschitzian at $x$ that 
\begin{equation*}
\partial_C(\sum f_i(x))\subseteq\sum\partial_C f_i(x).
\end{equation*}
If all these functions are convex, extended-valued or $C^{2}$ in the Fr\'echet sense, equality holds.  
\end{proposition}

We note that for every $x\in S$, we have:
\begin{equation*}
\partial_C I_S (x):=N^C_S (x).
\end{equation*}
and that if $I_{S}(x)<\infty$ and $S$ has nonempty interior, then $I_{S}$ is directionally Lipschitz at $x$, as noted above.  Also, as noted in \cite{Roc79}, if $S$ is convex, $I_{S}$ is subdifferentially regular wherever it is finite.  Therefore, in particular, we have
\begin{equation*}
\partial_{C}I_{[a,b]}(x)=N^{C}_{[a,b]}(x)=\begin{cases}
0&\text{if } a<x<b\\
(-\infty,0]&\text{if } x=a\\
[0,+\infty)&\text{if } x=b\\
\end{cases}
\end{equation*}
Again, this is not the strictest set of conditions for the subdifferential regularity of $I_{S}$, but they suffice for our purposes here.  Thus, if we have constraints requiring the result of a strictly differentiable operator applied to the state remains in a convex set with nonempty interior of a Hilbert space, we may make use of the above calculus rules with equality instead of merely inclusion.  

Finally, we note the following fact about descent directions and subdifferentials.  
\begin{proposition}
Let $f:X\rightarrow \R\cup\{+\infty\}$ be Lipschitz and $f(x)<\infty.$ Then 
\begin{equation*}
-\argmin_{\zeta\in\partial_{C}f(x)}||\zeta||
\end{equation*}
is a descent direction at $x.$
\begin{proof}{\em (Sketch)}
We have
\begin{align*}
-\argmin_{\zeta\in\partial_C f(x)} ||\zeta||&=-\arg(\min_{\zeta\in\partial_C f(x)}\max_{||d||\leq 1}\langle\zeta,d\rangle)\\
&=-\arg(\max_{||d||\leq 1}\min_{\zeta\in\partial_C f(x)}\langle \zeta,d\rangle)\\
&=-\arg(\max_{||d||\leq 1}\min_{\zeta\in\partial_C f(x)}\langle \zeta,-d\rangle).\\
&=\arg(\min_{||d||\leq 1}\max_{\zeta\in\partial_{C}f(x)}\langle \zeta,d\rangle).
\end{align*}
which means that the directional derivative is minimized, as the directional derivative in direction $d$ is 
equal to 
\begin{equation*}
\max_{\zeta\in \partial_{C}f(x)}\langle \zeta,d\rangle.
\end{equation*}
For a more detailed discussion of descent directions see \cite{Cla90}.
\end{proof}
\end{proposition}

\section{Optimality Conditions}
\label{sec:optimality}
With the above calculus rules, we turn to deriving conditions for the optimality of a control in the context of the three problems discussed in Section \ref{sec:problem}.  We note that, as $\mathcal{E}$ is linear and bounded, the first two terms of $j_{B},~ j_{WI},$ and $j_{R}$ are twice continuously Fr\'echet differentiable.  We denote by $\mathcal{E}^{*}:L^{2}(\Omega)\rightarrow L^{2}(\Omega)$ as the adjoint solution operator.  In the case of the box constrained problem, we may use the chain rule and sum rule to obtain the following.  
\begin{proposition}
\label{prop:boxopt}
The function $q^{*}\in L^{2}(\Omega)$ is a global minimizer of $j_{B}$ if and only if there exists a $\psi^{*},\lambda^{*}\in L^{2}(\Omega)$ such that
\begin{eqnarray*}
A\psi^{*}-q^{*}&=0\\
A^{*}\lambda^{*}+\big(\psi^{*}-\overline{\psi}\big)+N^{C}_{[a,b]}(\psi^{*})&=0\\
\alpha q^{*}+\lambda^{*}&=0.
\end{eqnarray*}
\end{proposition}
\remark We note that this agrees with, for instance, Theorem 6.5 of \cite{Tro02}.  

\begin{proof}
We use a standard Lagrange function construction to obtain the relevant optimality condition.  Here the Lagrange function $L_{B}:L^{2}(\Omega)\times L^{2}(\Omega)\times L^{2}(\Omega)\rightarrow \R\cup+\{\infty\}$ is of form
\begin{equation*}
L_{B}(\psi,q,\lambda):=\frac{1}{2}||\psi-\overline{\psi}||_{2}^{2}+\frac{\alpha}{2}||q||_{2}^{2}+\int_{\Omega} I_{[a,b]}(\psi)(x)dx-\int_{\Omega}\lambda(A\psi-q)dx.
\end{equation*}
This is clearly Clarke differentiable (in the extended-sense above) and so, for optimality to hold, we have the following necessary condition for optimality involving the partial Clarke derivatives of the Lagrange function (noting the conditions for equality in the above sum rule hold here):
\begin{eqnarray*}
A\psi^{*}-q^{*}&=0\\
-A^{*}\lambda^{*}+\big(\psi-\overline{\psi}\big)+\partial_{C_{\psi}}\big(\int_{\Omega}I_{[a,b]}(\psi)(x)dx\big)&=0\\
\alpha q^{*}-\lambda^{*}&=0.
\end{eqnarray*}
Here, $\partial_{C_{\psi}}$ denotes the partial Clarke subdifferential with respect to $\psi.$  The final term in the left hand side of the second equation can be written (see, for instance, Theorem 2 of \cite{LoeZhe95}) as
\begin{equation*}
\partial_{C_{\psi}}\big(\int_{\Omega}I_{[a,b]}(\psi)(x)dx\big)=N^{C}_{[a,b]}(\psi(x)).
\end{equation*}
As $J_{B}$ is convex and $\mathcal{E}$ is linear and bounded, we have sufficiency of this condition via standard arguments (see \cite{Tro02}) for instance.
\end{proof}
We characterize the global minimizers for the weighted integral constrained problem:
\begin{proposition}
The function $q^{*}\in L^{2}(\Omega)$ is a global minimizer of $j_{WI}$ if and only if
\begin{equation*}
0\in \big\{\mathcal{E}^{*}(\mathcal{E}q^{*}-\overline{\psi})+\alpha q^{*}\big\}+\Big(N^{C}_{[a,b]}(\langle w,\mathcal{E}q^{*}\rangle\Big)\cdot\mathcal{E}^{*}(w).
\end{equation*}
\end{proposition}
\begin{proof}
The arguments in the previous proof regarding the necessity and sufficiency of $0\in\partial_{C}j_{WI}(q^{*})$ hold here as well, as does the application of the sum rule.  The primary difference here is the chain rule calculation of the penalty term.  We note the Clarke subgradient of this term is a dual element which can be written in weak form,  for a test function $\psi\in L^{2}(\Omega),$ as
\begin{align*}
\int_{\Omega}\partial_{C}(I_{[a,b]}(\langle w,\mathcal{E}q^{*})\rangle )\psi dx &=\int_{\Omega}\Big[\big(D_{s}[\langle w,\mathcal{E}(q^{*})\rangle\big)^{*}\circ N^{C}_{[a.b]}(\langle w,\mathcal{E}q^{*}\rangle)\Big]\psi dx\\
&=\int_{\Omega}\mathcal{E}^{*}(w\cdot N^{C}_{[a.b]}(\langle w,\mathcal{E}q^{*}\rangle)\psi dx\\
&=\int_{\Omega}\mathcal{E}^{*}(w)\cdot N^{C}_{[a,b]}(\langle w,\mathcal{E}q^{*})\rangle)\psi dx
\end{align*} 
which provides the remaining ingredient for the proof.
\end{proof}
Finally, under additional assumptions on $\mathcal{E}$ and the optimal control $q^{*}$, we have the necessary optimality conditions for $j_{TC}:$
\begin{proposition}
\label{prop:optjtc}
Suppose that the function $q^{*}\in L^{2}(\Omega)$ is a local minimizer of $j_{TC},$ and
\begin{enumerate}
\item  $Z \subset \subset \Omega$ has $C^{1}$ boundary and nonempty interior, 
\item $\mathcal{E}$ is surjective on the space of extensions of functions in $C_{c}^{\infty}(Z)$,  
\item $\mathcal{E}(q^{*})$ is continuous on $Z.$
\end{enumerate}
Define $f:L^{2}(\Omega)\rightarrow L^{2}(\Omega)$ given by
\begin{equation*}
f(q):=\int_{Z}\big[\mathds{1}_{(-\infty,a)\cup(b,\infty)}\big(\mathcal{E}q\big)\big](x)dx
\end{equation*}
 either 
\begin{itemize}
\item $f(q^{*})<(1-c)|Z|$ and $-\alpha q^{*}=\mathcal{E}^{*}(\mathcal{E}q^{*}-\overline{\psi}),$ or
\item $f(q^{*})=(1-c)|Z|$ and for any desired accuracy $\epsilon>0,$ there exists $q_{\epsilon}$ s.t. $(q_{\epsilon},f(q_{\epsilon})\in (q^{*},f(q^{*}))+B_{\epsilon}$, with
\begin{equation*}
0\in \beta\big\{\mathcal{E}^{*}(\mathcal{E}q^{*}-\overline{\psi})+\alpha q^{*}\big\}+B^{w}_{\epsilon}+\gamma \partial_{P}f(q_{\epsilon})
\end{equation*}
where 
\begin{equation*}
\partial_{P}f(q_{\epsilon})(x)=\begin{cases}
[0,+\infty)\cdot\big[\mathcal{E}^{*}(\mathds{1}_{\R\setminus[a,b]}\mathcal{E}q_{\epsilon})\big](x) &\text{ if } [\mathcal{E}q_{\epsilon}](x)=b,\\
(-\infty,0]\cdot\big[\mathcal{E}^{*}(\mathds{1}_{\R\setminus[a,b]}\mathcal{E}q_{\epsilon})\big](x) &\text{ if } [\mathcal{E}q_{\epsilon}](x)=a,\\
0& \text{ else,}
\end{cases}
\end{equation*}
and $\beta+\gamma=1$ and $B^{w}_{\epsilon}$ denotes the $\epsilon$-ball in the weak topology of $L^{2}(\Omega).$
\end{itemize}
\end{proposition}
\begin{remark}
The assumption on the image $\mathcal{E}$ is motivated by the use of the method of manufactured solutions in applications.  In such applications, one may test solution methods by applying the relevant differential operator to a constructed solution in order to find a relevant control/source.  It is easy to find examples where it holds; one example is where $\mathcal{E}$ is associated to the Poisson problem with distributed source:
\begin{equation*}
-\Delta \psi=q.
\end{equation*}
If we instead have $\mathcal{E}$ associated to the Poisson problem with boundary control,
\begin{equation*}
-\Delta \psi=0,~\psi|_{\partial\Omega}=q,
\end{equation*}
then clearly, the relevant assumptions are not satisfied in the most general of settings.  We also are not sure how strict the assumption on the continuity of the optimal state is in a general setting.  However, it is possible that further information on the specific form of $Z$ and $\Omega$ will allow for this assumption to be unnecessary.
\end{remark}
\begin{remark}
The ``fuzzy'' nature of this rule may seem an obstacle at first glance to implementation.  However, if we discretize the problem with appropriate error bounds, we may require $\epsilon$ to be sufficiently small to be dominated by the discretization error.  Furthermore, the ``fuzzy'' nature of the optimality condition can be improved in the presence of constraint qualifications being satisfied. See \cite{BorZhu99} for a survey of some appropriate constraint qualifications.
\end{remark}
\begin{proof}{\em (of Proposition \ref{prop:optjtc})}
For simplicity, we only consider the case where $a=-\infty.$ Then we have 
\begin{equation*}
f(q):=\int_{Z}[\mathds{1}_{(b,\infty)}(\mathcal{E}q)](x)dx
\end{equation*}
If $f(q^{*})<(1-c)|Z|,$ it is clear that the result holds as the constraint is inactive and we recover the classical optimality conditions.  Assume that $f(q^{*})=(1-c)|Z|.$
We know that the following set has positive measure:
\begin{equation*}
Z^{b}:=\{x\in Z: \mathcal{E}q^{*}(x)\in(b,\infty)\}
\end{equation*}
and thus contains a nonempty open set. So there exists an open set of nonnegative smooth functions $\{\phi_{\gamma}\}$ with support contained in $Z^{b}.$   By assumption there is a set of corresponding controls $p_{\gamma}$ such that $\phi_{\gamma}=\big(\mathcal{E}p_{\gamma}\big)|_{Z}$.  By construction, $f(q+tq_{\gamma})=f(q)$ for all $t>0.$  Therefore there is an open set of directions such that $f$ at $q$ is directionally Lipschitz along them. 

Then, we note that if $\zeta\in\partial_{P}f(q)$ then almost everywhere in $\Omega$ we have that
\begin{equation*}
\zeta(x)\in\Big[\mathcal{E}^{*}\big(\mathds{1}_{(b,\infty)}\mathcal{E}q\big)\Big]\cdot\big[ \partial_{P}\mathds{1}_{(b,\infty)}\mathcal{E}q(x)\big];
\end{equation*}
where we note that differentiation under the integral is allowed for proper lower semicontinuous integrands in $L^{2}(\Omega)$ as discussed in \cite{OutRom05} and \cite{IofRoc96} and that, as above, $\mathcal{E}$ is regular in the sense of Clarke.  It is immediately seen that
\begin{equation*}
 \partial_{P}\mathds{1}_{(b,\infty)}\mathcal{E}(q)(x)=
\begin{cases} 
[0,+\infty) & \text{ if }[\mathcal{E}q](x)=b,\\
0 & \text{ else.}
\end{cases}
\end{equation*}
The rest of the proof follows from a lemma which is the following particular case of Theorem 2.14 from \cite{BorTreZhu98}.
\begin{lemma}
Assume that $q^{*}$ is a local minimizer of $j_{TC}.$  Then for any number $\epsilon >0,$ and weak neighborhood $V$ of $0,$ there exists $\tilde{q}$ such that  $\big(\tilde{q},f(\tilde{q})\big)\in (q^{*},f(q^{*})\big)+\epsilon B$ and
\begin{equation*}
0\in\beta\mathcal{E}^{*}(\mathcal{E}(q^{*}-\overline{\psi}))+\gamma\partial_{P}f(\tilde{q})+V
\end{equation*}
where $\beta+\gamma=1.$
\end{lemma}
The referenced Theorem is in a much more general setting, and relies on a ``fuzzy'' sum rule (see also \cite{ClaLedSte98} for a discussion of such rules).  Here we take advantage of the smoothness of the tracking functional to reduce the ``fuzziness'' of that rule to some extent.  From there the result follows immediately.  
\end{proof}

\section{Numerical Implementations}
\label{sec:optalg}

\subsection{Subgradient descent method} \label{subsec:steep} We consider a steepest descent algorithm for solving the problems by noting that the element subgradient of minimal norm is a direction of descent.  
\begin{algorithm}
We initialize with a feasible point $q^{(0)}.$
\begin{enumerate}
\item At $q^{(k)},$ we calculate the state $\psi^{(k)}:=\mathcal{E}(q^{(k)})$ and the classical adjoint $\lambda^{(k)}=\mathcal{E}^{*}(\psi^{(k)}-\overline{\psi}).$
\item We then minimize $\frac{1}{2}||\rho+\lambda^{(k)}+\alpha q^{(k)}||_{2}^{2}+\frac{\beta}{2}||\rho||^{2}_{2}$ over $\rho$ and label the minimizer $\rho^{(k)}$ for a very small $\beta>0.$  If the minimum norm is below tolerance, {\em STOP}.  The minimization is constrained depending on the problem:
\begin{enumerate}
\item For the box constrained problem, $\rho(x)\in N_{[a,b]}(\psi^{(k)})(x),$
\item For the weighted integral constrained problem, $\rho(x)=r\cdot\mathcal{E}^{*}(\psi^{(k)})$ for $r\in N_{[a,b]}(\langle w,\psi^{(k)}\rangle),$
\end{enumerate}
\item Perform an Armijo line search along the direction $d^{(k)}:=-\rho^{(k)}-\lambda^{(k)}-\alpha q^{(k)}$ to obtain $q^{(k+1)}.$ If resulting step size is below tolerance, {\em STOP}.
\end{enumerate}
\end{algorithm}
For the total coverage constrained problem, we consider for notational simplicity only the case where we have discretized the problem.  Then the minimization in step (2) is performed only if the constraint is active at the current iterate and is taken over \begin{equation*}
\rho_{h}(x)\in -[\mathcal{E}^{*}_{h}(\mathds{1}_{\R\setminus[a,b]}\psi_{h}^{(k)})](x)\cdot N_{[a,b]}(\psi_{h}^{(k)})
\end{equation*}
where we use the subscript $h$ to emphasize this is taken in the discrete setting (that is, $\mathcal{E}_{h}$ is the discretized adjoint operator).  Note that the direction search is only performed if the constraint is active and that for the case of the weighted integral constraint, it is an additional line search.  

\begin{remark}
As noted in the Introduction, significant attention has been paid to efficient numerical methods for box constrained problems; specifically, semismooth Newton Methods and smoothing strategies \cite{HintermullerUlbrich2004, HintermullerItoKunisch2003, GugatHerty} have demonstrated better convergence rates than might be expected from the above method.  We include it for completeness. 
\end{remark}

\subsection{Numerical Example}
\label{sec:numerics}
In this section, we focus on test problems where the solution operator $\mathcal{E}$ is associated to solving the Poisson equation subject to homogeneous Dirichlet boundary conditions.
\begin{eqnarray*}
-\Delta \psi&=q\\
\psi|_{\partial\Omega}&=0. 
\end{eqnarray*}
We test the above algorithm where the desired state is
\begin{equation*}
\overline{\psi}_{1}:=\mathcal{E}(e^{-|x|^{2}/4})
\end{equation*}
where $\Omega$ the unit circle; it is then rescaled into a unit $L^{2}$ function.  We use $1.e-5$ as tolerance for both convergence tests described above and set $\alpha=.001$.  We solve  the weighted integral constrained problem with $w$ as the characteristic function associated with the $0.25$ radius ball using the descent method described above.  The upper bound is set to $0.12.$  As a reference, $\langle w,\overline{\psi}\rangle_{L^{2}(\Omega)}\approx0.1817$ in this example.  By encoding the constraint via infinite penalization, we achieve significant speedup with the number of PDE solves being independent of mesh while capturing the minimum more closely than the SQP.  One such result is displayed in Figure \ref{fig:avcirc08}, along with a contour plot of the result to show the effect of the constraint.   We repeat this for a problem where the source is the same, however the mesh is on an L-shaped domain, $w$ is the the characteristic of $[-0.5,0.5]\times [-0.25,-0.75]$ and the upper bound for the average value is set as $\frac{1}{2}\langle w,\overline{\psi}\rangle_{L^{2}(\Omega}.$  The relative counts for iterations and PDE solver calls are similar to those of the previous example; thus, we show the residual for the results of infinite penalization and SQP compared with the boundary of the support of $w$ along with the optimal result from infinite penalization.
\begin{table}
\begin{tabular}{| c | | c | c | c || c | c | c|}
\hline
& \multicolumn{3}  {|c||} {Inf. Pen. } & \multicolumn{3}{|c|}{SQP}\\
\hline
Elem \#. & Opt. Cost & Iters. & PDE Solves & Opt. Cost & Iters. & PDE Solves\\
\hline
377 & 0.0167& 6 & 446 & 0.021 & 14 & 5330\\
541 & 0.0263& 4 & 283 & 0.0262 & 14 & 7641\\
749 & 0.0197& 6 & 440 & 0.0220 & 14 & 10554 \\
1105  & 0.0217& 4 & 281 & 0.0316 & 10 & 11099\\
1445 & 0.0209& 4 & 268 & 0.0225 & 14 & 20300\\
2109 & 0.0212& 4 & 284 & 0.0226 & 9 & 19023\\
\hline
\end{tabular}
\caption{Weighted integral problem mesh dependence comparison}
\label{tab:WImesh}

\end{table}
\begin{figure}
\centering
\subfigure[Target]{
\label{fig:avtarg08}
\includegraphics [width=6cm]{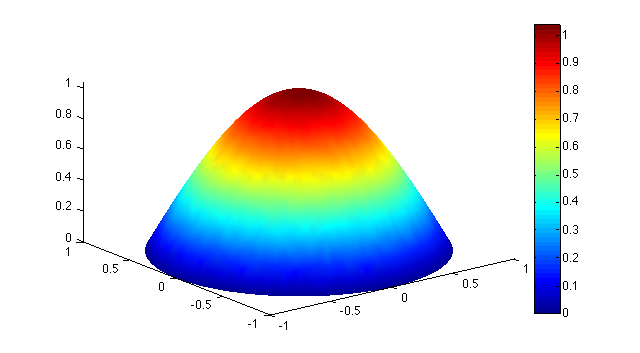}
}
\subfigure[Result of infinite penalization]{
\includegraphics [width=6cm]{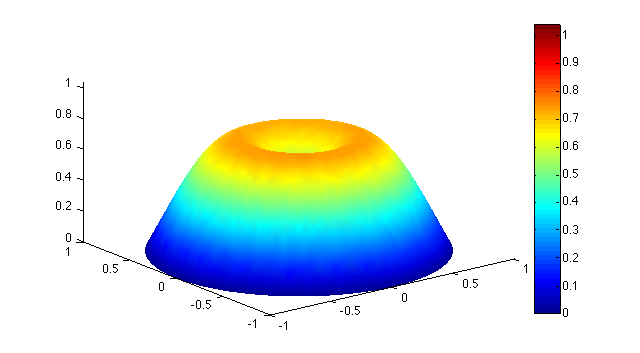}
}\\
\subfigure[Contours of result]{
\includegraphics[width=6cm]{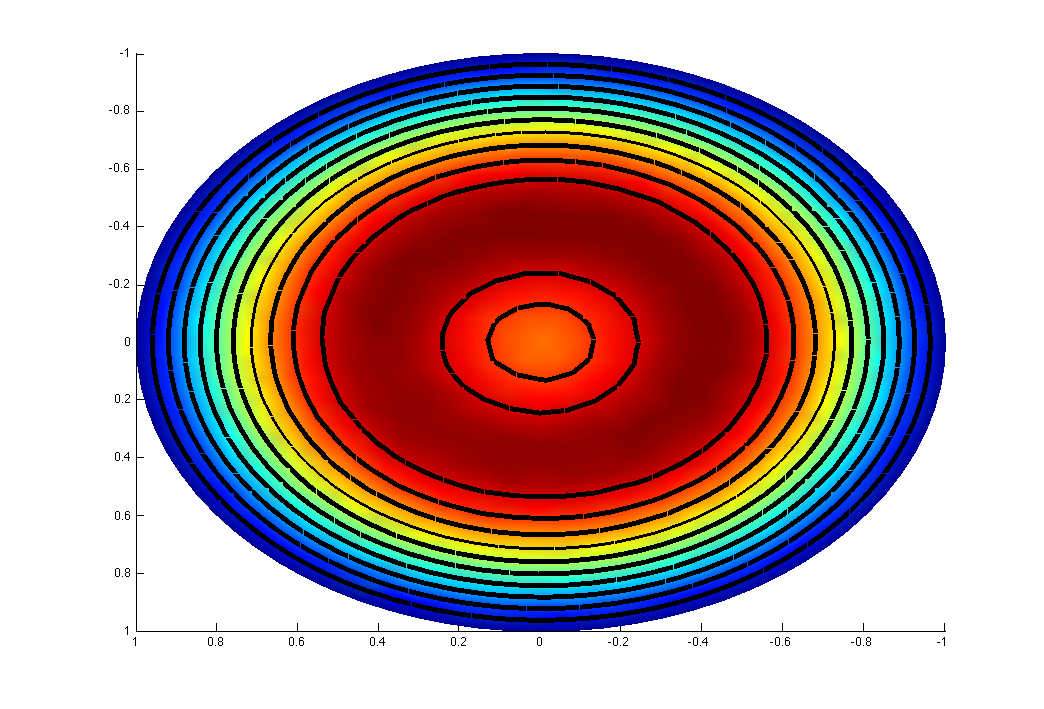}
}
\caption{Weighted integral constrained target and result from gradient descent using infinite penalization}
\label{fig:avcirc08}
\end{figure}

\begin{figure}
\centering
\subfigure[Target]{
\label{fig:boxtarg08}
\includegraphics [width=6cm]{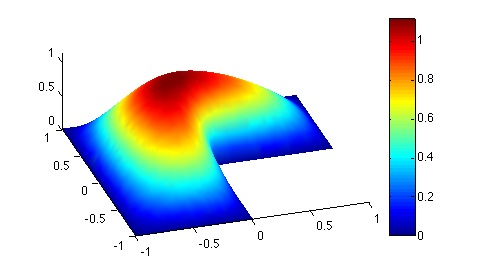}
}
\subfigure[Result of infinite penalization]{
\includegraphics [width=6cm]{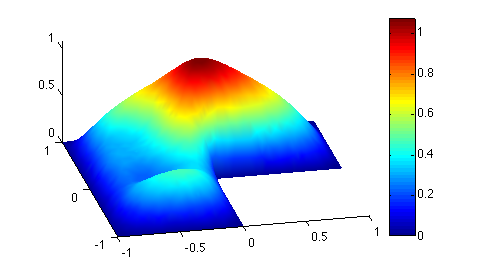}
}\\
\subfigure[Infinite Penalization $|\psi^{*}-\overline{\psi}|$\newline with boundary of support of $w$ in white]{
\includegraphics[width=6cm]{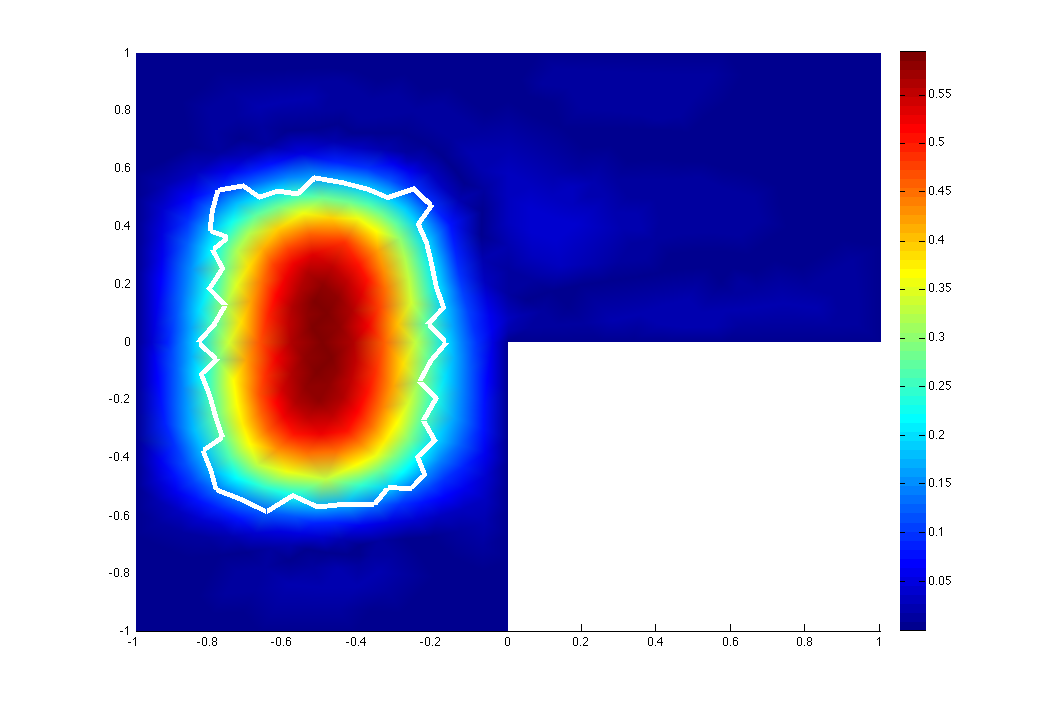}
}
\subfigure[SQP $|\psi^{*}-\overline{\psi}|$ with boundary of support of $w$ in white]{
\includegraphics[width=6cm]{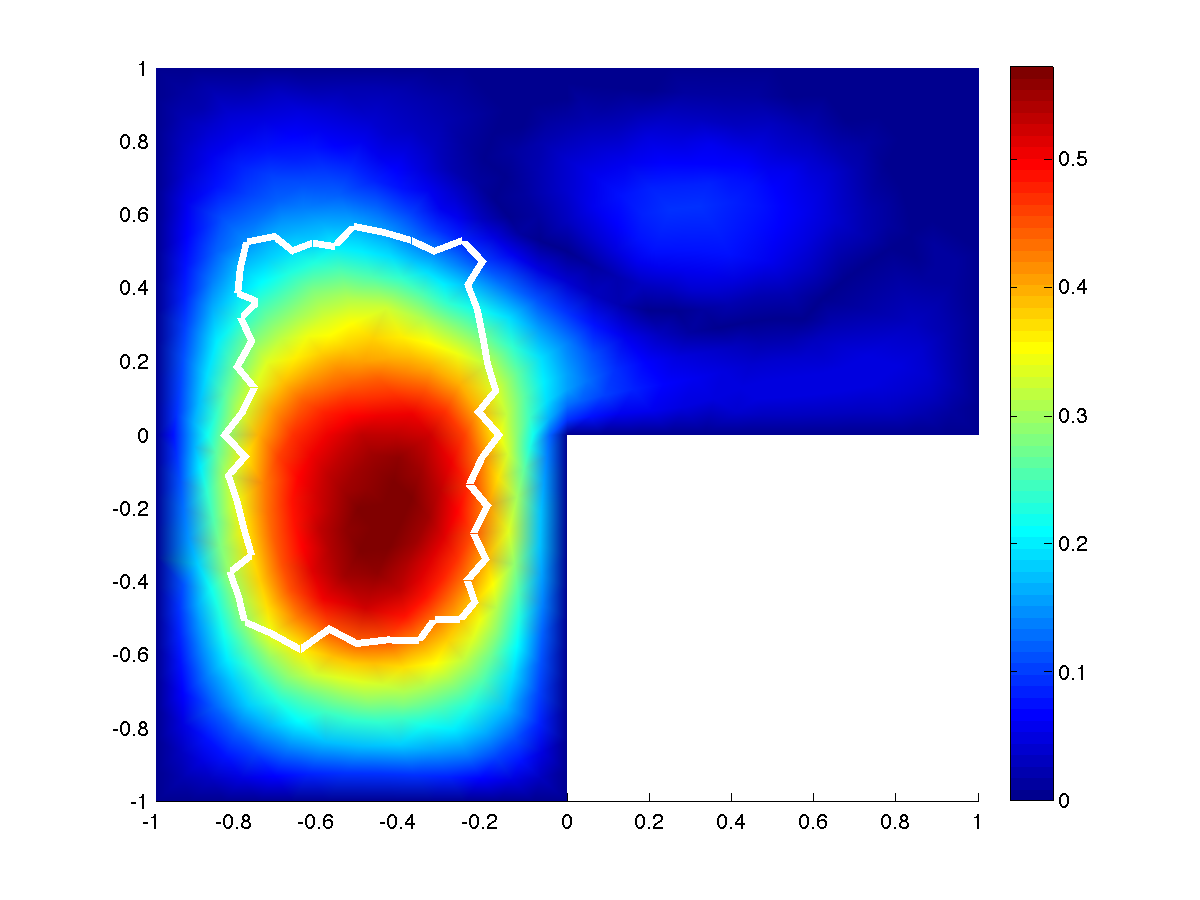}
}
\caption{Weighted integral constrained target and result from BFGS  using infinite penalization}
\end{figure}

\section{Concluding Remarks}
We have considered a collection of primal-based methods in the absence of constraint qualifications for deriving necessary optimality conditions for optimal control problems involving PDEs whose solution operator is linear and bounded.  The framework used is rather flexible, and can be extended to cases where the operator $\mathcal{E}$ is strictly differentiable, as opposed to linear and bounded.  Furthermore, other constraints can be implemented, as long as they satisfy the appropriate regularity in order to make use of the appropriate chain rule.  Using these conditions, we implemented a steepest descent algorithm significantly performs better than a black box SQP solver.  At this stage, it is unclear to what extent fully optimizing the norm of the subgradient is required for descent; also it may be possible to use nonsmooth second-order rules to develop a robust Newton-like method.

\bibliographystyle{plain}
\bibliography{artikel2,references}

\end{document}